\documentclass[12pt]{article}
\usepackage{e-jc}

\usepackage{amsmath,graphicx}
\usepackage[T1]{fontenc}
\usepackage[utf8]{inputenc}

\theoremstyle{plain}
\newtheorem{thm}{Theorem}
\newtheorem{prop}{Proposition}
\newtheorem{lem}{Lemma}
\newtheorem{cor}{Corollary}
\theoremstyle{remark}
\newtheorem{rem}{Remark}

\dateline{Jan 6, 2021}{Jun 25, 2021}{TBD}


\Copyright{The authors. Released under the CC BY-ND license (International 4.0).}

\title{Enumeration of planar constellations with an alternating
  boundary}

\author{Jérémie Bouttier\\
  \small Université Paris-Saclay, CNRS, CEA, Institut de physique théorique, 91191, Gif-sur-Yvette, France \\[-0.8ex]
  \small Univ Lyon, Ens de Lyon, Univ Claude Bernard, CNRS, Laboratoire de Physique, F-69342 Lyon\\[-0.8ex]
  \small\tt jeremie.bouttier@ipht.fr\\
  \and
  Ariane Carrance\thanks{Supported by the ERC Advanced Grant 740943 GeoBrown.}\\
  \small Université Paris-Saclay, CNRS, Laboratoire de Mathématiques d'Orsay, F-91405 Orsay, France\\[-0.8ex]
  \small\tt ariane.carrance@math.cnrs.fr}

\def\volno{28(3)}
\def\volyear{2021}
\def\papno{P3.21}
\def\papid{10149}

\begin{document}

\maketitle


\begin{abstract}
  A planar hypermap with a boundary is defined as a planar map with a
  boundary, endowed with a proper bicoloring of the inner faces. The
  boundary is said alternating if the colors of the incident inner
  faces alternate along its contour. In this paper we consider the
  problem of counting planar hypermaps with an alternating boundary,
  according to the perimeter and to the degree distribution of inner
  faces of each color. The problem is translated into a functional
  equation with a catalytic variable determining the corresponding
  generating function.

  In the case of constellations---hypermaps whose all inner faces of a
  given color have degree $m\geq 2$, and whose all other inner faces
  have a degree multiple of $m$---we completely solve the functional
  equation, and show that the generating function is algebraic and
  admits an explicit rational parametrization.

  We finally specialize to the case of Eulerian
  triangulations---hypermaps whose all inner faces have degree
  $3$---and compute asymptotics which are needed in another work by
  the second author, to prove the convergence of rescaled planar
  Eulerian triangulations to the Brownian map.
\end{abstract}

\section{Introduction}

\paragraph{Context and motivations.}

The enumerative theory of planar maps has been an active topic in
combinatorics since its inception by Tutte in the
sixties~\cite{Tutte68}.  A recent account of its developments may be
found in the review by Schaeffer~\cite{Schaeffer15}. Many such
developments were motivated by connections with other fields:
theoretical physics, algebraic and enumerative geometry, computer
science, probability theory... In this paper, we consider a specific
question of map enumeration which finds its origin in the
work~\cite{carrance-trig-eul} of the second author on the scaling
limit of random planar Eulerian triangulations. In fact, we establish
key asymptotic estimates which are needed in the proof that these maps
converge to the Brownian map.

Eulerian triangulations are part of a more general family of maps
which, following~\cite{bernardi-fusy-2020}, we call
\emph{hypermaps}. These are maps whose faces are bicolored (say, in
black and white) in such a way that no two faces of the same color are
adjacent. They correspond through duality to bipartite maps. A fairly
general enumeration problem consists in counting planar hypermaps with
a given degree distribution of faces of each color. This problem is
intimately connected~\cite{BoSc02_Ising} with the Ising model on
random maps, first solved by Kazakov and
Boulatov~\cite{Kazakov86,BoKa87}. It received a lot of attention in
the physics literature, due to its rich critical behavior and its
connection with the so-called two-matrix model~\cite{Douglas91}. A
nice account of the enumerative results coming from this approach may
be found in the last chapter of the book by Eynard~\cite{eynard}. Let
us also mention the many recent papers~\cite{ams,ct,ct2,turunen20}
devoted to the local limits of the Ising model on random maps.

To enumerate maps, the standard approach~\cite{Tutte68} consists in
studying the effect of the removal of an edge. Nowadays, this
operation is often called ``peeling''~\cite{CurienSFnotes}. In order
to turn the peeling into equations, one needs to keep track of one or
more auxiliary parameters, called catalytic variables~\cite{BoJe06},
and typically corresponding to boundary lengths. In the context of
hypermaps, a further complication occurs, since one needs to keep
track of \emph{boundary conditions}: colloquially speaking, these
encode the colors of the faces incident to the boundary. The most
tractable boundary condition is the so-called Dobrushin boundary
condition, which consists in having the boundary made of two parts:
one part will be incident to white faces, and the other to black
faces. The Dobrushin boundary condition has the key property that it
is invariant under peeling, provided that we always peel an edge at
the white-black interface. The resulting equations are solved
in~\cite{eynard,ct}, see also~\cite{angel} for a probabilistic
interpretation in the context of site percolation. Knowing how to
treat Dobrushin boundary conditions, one may then consider ``mixed''
boundary conditions~\cite{eynard}, where there is a prescribed number
of white-black interfaces. However, this approach seems to become
intractable when the number of interfaces gets large. Here, we are
considering the extreme situation where the number of interfaces is
maximal, that is to say when white and black faces alternate along the
boundary. We call this situation the \emph{alternating boundary
  condition}. It arises when considering the layer decomposition of
Eulerian triangulations~\cite{carrance-trig-eul}. The alternating
boundary condition is not invariant under peeling, but can be made so
by adding a monochromatic boundary part, and peeling at the first
alternating edge. Through this method, we obtain a functional equation
amenable to the kernel method~\cite{Prodinger04,BoJe06}. We do not
attempt to solve this equation in general, but we concentrate on the
specific case of \emph{constellations}\footnote{We follow the
  terminology of~\cite{bm-s}. Even though we view constellations as
  instances of hypermaps, they can also be viewed as more general
  objects~\cite{lando-zvonkine,BC20}, but we do not consider this
  setting here.}  (encompassing Eulerian triangulations). Thanks to a
shortcut, we directly obtain a rational parametrization of the
generating function of constellations with an alternating boundary,
and having this parametrization at hand is very convenient to compute
asymptotics.

Before presenting our results in more detail, let us mention that
there exists several bijective approaches to the enumeration of
hypermaps~\cite{BoSc02_Ising,bdg,fomap,bouttier-fusy-guitter,bernardi-fusy-2020,AB21},
but we have not succeeded so far in adapting them to the case of
alternating boundary conditions. We leave this as an open problem.

\paragraph{Main results.} Recall that a \emph{planar map} is a
connected multigraph drawn on the sphere without edge crossings, and
considered up to continuous deformation. It consists of vertices,
edges, faces, and corners. A~\emph{boundary} is a distinguished face, which we
often choose as the infinite face when drawing the map on the
plane. It is assumed to be~\emph{rooted}, that is to say there is a
distinguished corner along the boundary. The other faces are called
\emph{inner faces}.

\begin{figure}[h]
\centering
\includegraphics[scale=0.5]{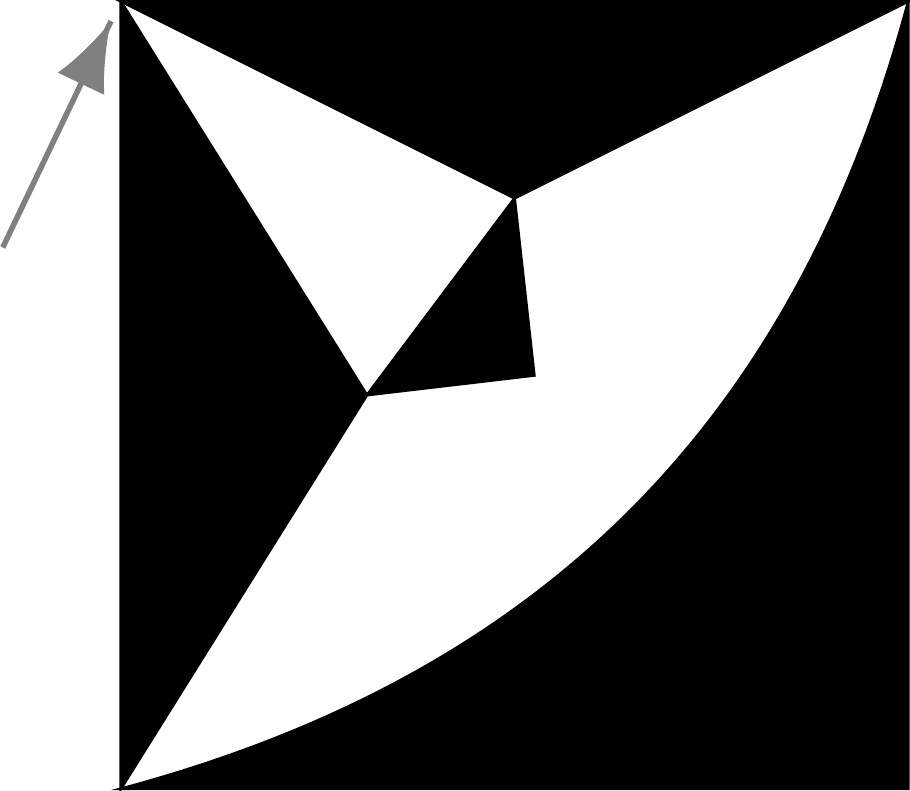}
\hspace{6em}
\includegraphics[scale=0.5]{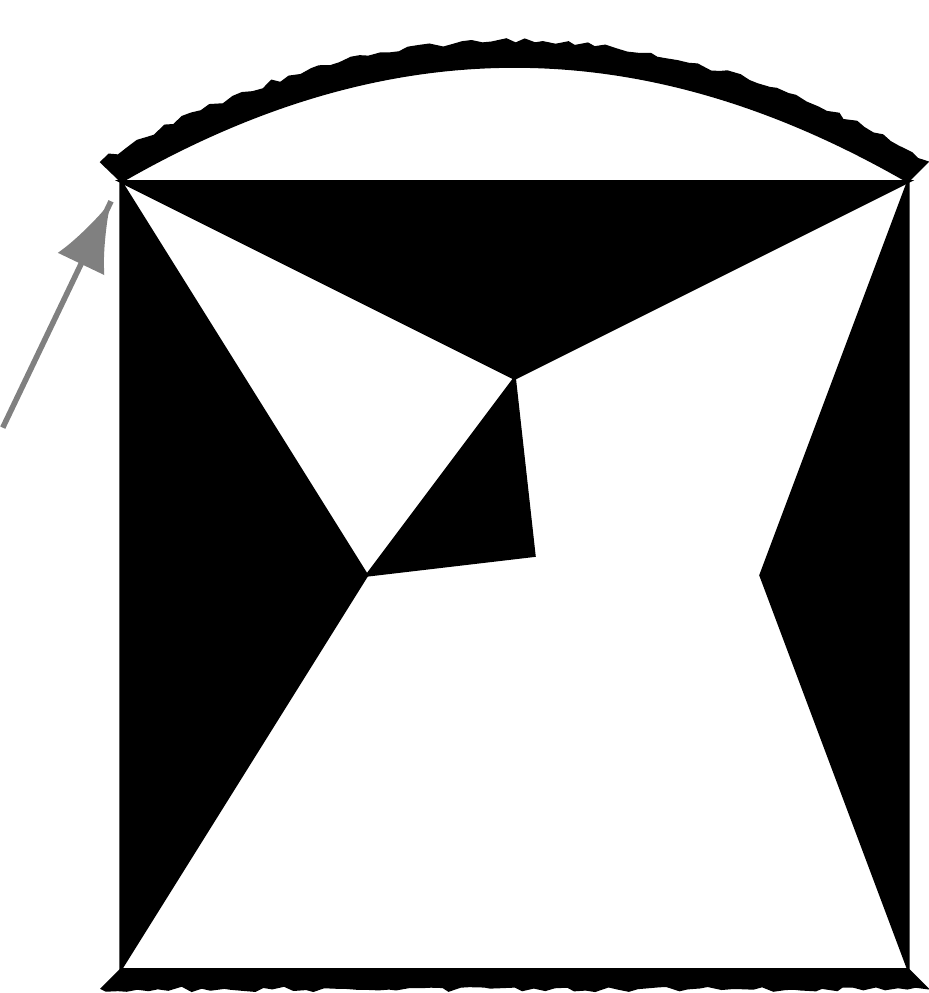}
\caption{Left: a hypermap with a white monochromatic boundary
  ($w=\mathord{\circ\circ}\mathord{\circ\circ}$), with its root corner
  marked by a gray arrow. Right: a hypermap with an alternating
  boundary ($w=\mathord{\circ\bullet}\mathord{\circ\bullet}$). We
  paint in black the external sides of the boundary edges with
  $w_i=\bullet$, since the boundary is considered black at these
  edges.}
\label{fig:HypermapEx}
\end{figure}

A~\emph{hypermap with a boundary} is a planar map with a boundary,
where every inner face is colored either black or white, in a such a
way that adjacent inner faces have different colors. Note that the
boundary does not have a specified color: it may be incident to faces
of both colors. Let $w$ be a word on the alphabet $\{\bullet,\circ\}$,
and $\ell$ be the length of $w$. A \emph{hypermap with boundary
  condition $w$} is a hypermap with a boundary of length $\ell$ such
that, when walking in clockwise direction around the map (starting at
the root corner),
\begin{enumerate}
\item if the $i$-th edge that we visit ($i=1,\ldots,\ell$) is incident
  to an inner face, then this face is white if $w_i=\bullet$ and black
  if $w_i=\circ$,
\item if an edge is incident to the boundary on both sides, and is
  then visited twice, say at steps $i$ and $j$, then $w_i \neq w_j$.
\end{enumerate}
In other words, $w_i$ codes for the color of the boundary face as seen
from the $i$-th edge.  A \emph{hypermap with a white (resp.\ black)
  monochromatic boundary} is a hypermap with boundary condition
$\circ\circ\cdots\circ$ (resp.\ $\bullet\bullet\cdots\bullet$). These
are just hypermaps in the sense of~\cite{bernardi-fusy-2020}. A
\emph{hypermap with an alternating boundary} is a hypermap with
boundary condition $\circ\bullet\circ\bullet\cdots\circ\bullet$. See
Figure~\ref{fig:HypermapEx} for an illustration.

Let $m$ be an integer not smaller than $2$. A hypermap (with a
boundary) is called an~\emph{$m$-constellation} if every black (inner)
face has degree $m$, and every white (inner) face has degree multiple
of $m$. The enumeration of $m$-constellations with a monochromatic
boundary and a prescribed degree distribution for the white faces has
been considered in previous works~\cite{bm-s,albenque-bouttier}, and
we review the results in Section~\ref{sec:reminders}. Our main result
concerns the enumeration of $m$-constellations with an alternating
boundary:

\begin{thm}
  \label{thm:main}
  Given a nonnegative integer $r$, let $A_r$ denote the generating
  function of $m$-constellations with an alternating boundary of
  length $2r$, counted with a weight $t$ per vertex and a weight $x_i$
  per white inner face of degree $mi$ for every $i \geq 1$. By
  convention, we have $A_0=t$, corresponding to the vertex-map. Let
  $d$ be a positive integer and assume that $x_i=0$ for $i>d$. Then,
  the series
  \begin{equation}
    \label{eq:Adef}
    A(w) := \sum_{r \geq 0} A_r w^{r+1}
  \end{equation}
  is algebraic and admits a rational parametrization: for $s$ a formal variable we have
  \begin{equation}
    A(w(s)) =  1 - \frac{(1+\sum_{k=1}^d \alpha_k s^k)^{m-1}}{1+V s}, \qquad
    w(s) := s \frac{(1+\sum_{k=1}^d \alpha_k s^k)^{m-2}}{(1+V s)^2}
  \end{equation}
  where $V$ is the series in $t,x_1,\ldots,x_d$ such that
  \begin{equation}
    \label{eq:Vdef}
    V = t + \sum_{i = 1}^d \binom{mi-1}{i} x_i V^{(m-1)i}
  \end{equation}
  and
  \begin{equation}
    \label{eq:alphadef}
    \alpha_k := \sum_{i=k}^d \binom{mi-1}{i-k} x_i V^{(m-1)i+k-1}, \qquad k=1,\ldots,d.
  \end{equation}
\end{thm}

Note that $w(s)$ is a series in $s$ whose coefficients are polynomials
in $V,\alpha_1,\ldots,\alpha_d$, with $w(s)=s+O(s^2)$. Therefore, the
substitution $A(w(s))$ is well-defined and, by reversion of series, we
see that $A_r$ is a polynomial in $V,\alpha_1,\ldots,\alpha_d$, hence
in $V,x_1,\ldots,x_d$. In particular we have $A_0=V-(m-1)\alpha_1=t$
as wanted. We give the expression for $A_1$ in
Appendix~\ref{sec:rootconst}, and we also check that it is consistent
with the generating function of rooted $m$-constellations computed by
another approach. In principle, we may express $A_r$ via the Lagrange
inversion formula, but we do not expect to obtain a practical formula
in this way. Still, as we shall see below, the rational
parametrization contains all the information needed to compute
asymptotics.

In the case $m=2$, through a classical identification between
$2$-constellations and planar bipartite maps, $A_r$ should be the
generating function of planar bipartite maps with a boundary of length
$2r$. This is indeed the case, as we will see in
Remark~\ref{rem:bipcase} below.

We now focus on the case of \emph{Eulerian $m$-angulations} (with a
boundary), i.e.\ $m$-constellations whose white (inner) faces all have
degree $m$. By specializing Theorem~\ref{thm:main} to $d=1$ and
$x_1=1$, we obtain:

\begin{cor}
  \label{cor:mangpar}
  For (planar) Eulerian $m$-angulations with an alternating boundary,
  counted with a weight $t$ per vertex, the rational parametrization
  takes the form
  \begin{equation}
    A(w(s)) =  1 - \frac{(1+V^{m-1} s)^{m-1}}{1+V s}, \qquad
    w(s) := s \frac{(1+V^{m-1} s)^{m-2}}{(1+V s)^2}
  \end{equation}
  where $V$ is the series in $t$ such that
  $V=t+(m-1)V^{m-1}$.
\end{cor}

From this we may compute the first few terms of $A(w)$, namely
\begin{multline}
  \label{eq:Awexp}
  A(w) = t w + \left(V^2 + (2m-3)V^m + \frac{(m-1)(m-2)}2 V^{2m-2}\right) w^2 + \\
  \textstyle \left( 2 V^3 + (6m-10) V^{m+1} + (m-2)(4m-5) V^{2m-1} - \frac{(m-1)(m-2)(2m-3)}3 V^{3m-3} \right) w^3 + \cdots,
\end{multline}
and the coefficient of $w^r$ is a polynomial in $V$. Note that we may
apply again Lagrange inversion, this time on the variable $t$, to get
that
$V^k = \sum_{n \geq 0} \frac{k(m-1)^n}{mn+k} \binom{mn+k}{n}
t^{(m-1)n+k}$ for any $k \geq 1$. Plugging this
into~\eqref{eq:Awexp} allows to extract the coefficient of $t^n w^r$
in $A(w)$.

\begin{figure}[h]
  \centering
  \includegraphics[width=\textwidth]{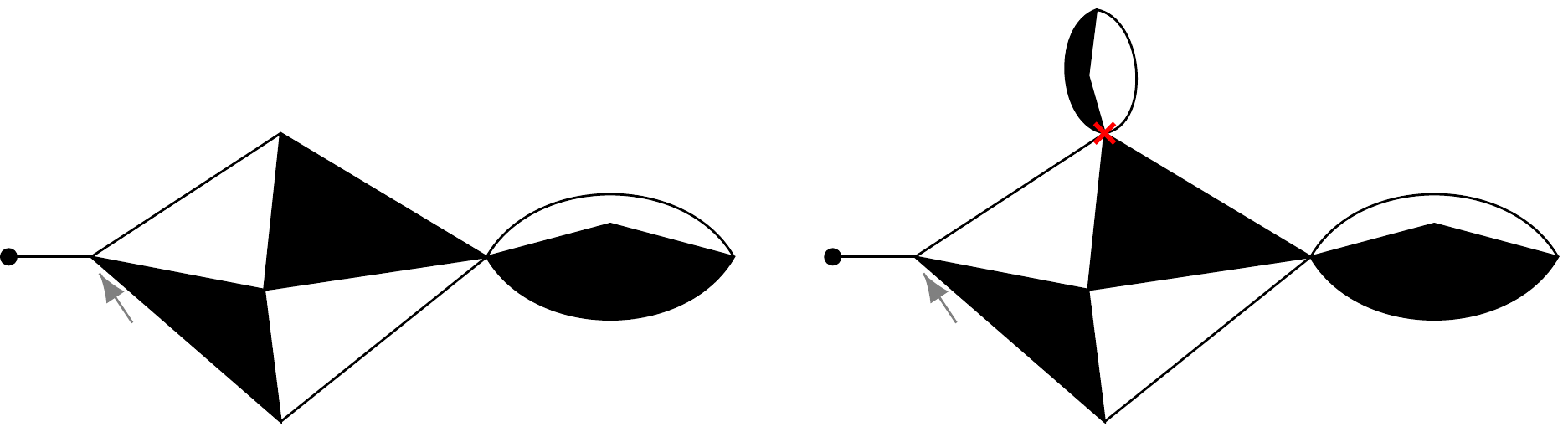}
  \caption{Left: an Eulerian triangulation with a semi-simple
    alternating boundary. Right: an Eulerian triangulation with an
    alternating, but not semi-simple boundary: the vertex that
    violates the semi-simplicity condition is indicated with a red
    cross.}
\label{fig:semi-simple}
\end{figure}

Let us emphasize that the above results concern maps whose boundary is
not necessarily simple. It is however not difficult to deduce a
rational parametrization for maps with a simple boundary via a
substitution argument. In fact, the layer decomposition introduced
in~\cite{carrance-trig-eul} involves Eulerian triangulations with a
\emph{semi-simple} alternating boundary: the external face can have
separating vertices, but only on every other boundary vertex. More
precisely, if we number the boundary corners $0,1,\ldots,2r-1$ in
clockwise direction, $0$ being the root corner, then the boundary is
said semi-simple if there is no separating vertex incident to an
odd-numbered corner.  See Figure~\ref{fig:semi-simple} for an
illustration.  Note that this definition allows for separating
boundary edges, but such edges are necessarily ``pendant'', i.e.\
incident to a vertex of degree one.  From the case $m=3$ of
Corollary~\ref{cor:mangpar} and through substitution and singularity
analysis, we obtain the following results, needed
in~\cite{carrance-trig-eul}:

\begin{thm}
  \label{thm:asymp}
  For $n,p$ nonnegative integers, denote by $B_{n,p}$ the number of
  Eulerian triangulations with a semi-simple alternating boundary of
  length $2p$ and with $n$ black triangles.
  Then, the corresponding generating function reads
  \begin{equation}
    \label{eq:Bsol}
    B(t,z) := \sum_{n \geq 0} \sum_{p \geq 0} B_{n,p} t^n z^p = 
    \frac1{16t} \left( 8t+8z+ \left(1-\sqrt{1-8t}-4z\right)
      \textstyle \sqrt{\frac{(1+\sqrt{1-8t})^2-4z}{1-z}}\right).
  \end{equation}
  For fixed $z \in[0, \frac{1}{4})$, we have the asymptotics
  \begin{equation}
    \label{eq:alternating-gf-better-asymp}
    [t^n]B(t,z) =
    \sum_{p \geq 0} B_{n,p}z^p \underset{n \to \infty}{\sim} \frac{3}{2}\frac{z}{\sqrt{\pi(z-1)(4z-1)^3}}8^n n^{-5/2}
  \end{equation}
  and, for any fixed $p$,
  \begin{equation}
    \label{eq:alternating-gf-first-asymp}
    B_{n,p}\underset{n \to \infty}{\sim} C(p)8^n n^{-5/2},
  \end{equation}
  with 
  \begin{equation}
    \label{eq:alternating-gf-second-asymp}
    C(p)\underset{p \to \infty}{\sim} \frac{\sqrt{3}}{2 \pi}4^p \sqrt{p}.
  \end{equation}
  Moreover, the quantity
  $Z(p):=\sum_nB_{n,p}8^{-n}= B\left(1/8,z\right)$ is finite for all
  $p \geq 0$, and satisfies
  \begin{align}
    \label{eq:alternatingZ}
    \sum_{p \geq 0}Z(p)z^p & = \frac{1}{2}\left(1+8z+\frac{(1-4z)^{3/2}}{(1-z)^{1/2}}\right),\\
    Z(p)&\underset{p \to \infty}{\sim}\frac{1}{4}\sqrt{\frac{3}{\pi}}4^p p^{-5/2}, \qquad Z(0)=1.
  \end{align}
\end{thm}

Note that $B(0,z)=\frac{1}{1-z}$, i.e.\ $B_{0,p}=1$ for all
$p \geq 0$: this corresponds to the ``star'' formed by $p$ edges, the
root corner being incident to the central vertex.

The quantity $Z(p)$ can be understood as the partition function of the probability measure on all Eulerian triangulations of the disk having a semi-simple alternating boundary of length $2p$, that assigns a weight $8^{-n}/Z(p)$ to each such triangulation that has $n$ black triangles. In particular, a random Eulerian triangulation sampled according to this law has a fixed perimeter, but not a fixed size, and the variance of that size is infinite. By analogy with similar families of random maps, we therefore call $Z(p)$ the partition function of \emph{critical Boltzmann Eulerian triangulations} with a semi-simple alternating boundary of length $2p$.

\paragraph{Outline.} In Section~\ref{sec:reminders}, we review
enumerative results about $m$-constellations with a monochromatic
boundary.  In Section~\ref{sec:funeq}, we derive a functional equation
determining the generating function of hypermaps with an alternating
boundary of prescribed length. We explain in
Section~\ref{sec:constsimp} how to solve this equation in the case of
$m$-constellations, which leads to the proof of
Theorem~\ref{thm:main}. We consider Eulerian triangulations with a
semi-simple alternating boundary in Section~\ref{sec: eul trig}, and
establish Theorem~\ref{thm:asymp}. Concluding remarks are gathered in
Section~\ref{sec:conc}.

Auxiliary material is left in the appendices: in
Appendix~\ref{sec:rootconst} we compute the generating function of
rooted $m$-constellations via two approaches; Appendix~\ref{sec:ising}
deals with the rational parametrization of the generating function of
hypermaps with a monochromatic boundary of prescribed length.

\paragraph{Acknowledgments.} We thank Marie Albenque and Grégory Miermont for useful
discussions. The work of JB is partly supported by the Agence
Nationale de la Recherche via the grant ANR-18-CE40-0033 ``Dimers''. The work of AC is supported by the ERC Advanced Grant 740943 GeoBrown.

\section{Reminders on the enumeration of constellations}
\label{sec:reminders}

In this section we review enumerative results about $m$-constellations
with a monochromatic boundary. Note that an $m$-constellation with a
white monochromatic boundary is just a rooted constellation in the
classical sense, since the boundary is a white face just like the
others, except that one of its corners is distinguished.

\begin{prop}[{\cite[Theorem~3]{bm-s}}]
  Let $(n_i)_{i \geq 1}$ be a sequence of nonnegative integers such
  that $0<\sum n_i <\infty$.  Then, the number of rooted
  $m$-constellations having $n_i$ white faces of degree $mi$ for every
  $i \geq 1$ is equal to
  \begin{equation}
    \label{eq:bms}
    \frac{m}{m-1} \cdot \frac{(v+\sum n_i-2)!}{v!} \cdot
    \prod_{i \geq 1} \frac{1}{n_i!} \binom{mi-1}i^{n_i}
  \end{equation}
  where $v=\sum (mi-i-1)n_i +2$ is the number of vertices.
\end{prop}

The above formula is closely related to the series $V$ of
\eqref{eq:Vdef}. Indeed, by the Lagrange inversion formula, or
equivalently by the cyclic lemma, we have
\begin{equation}
  \label{eq:lagV}
  \left[x_1^{n_1} \cdots x_d^{n_d}\right] V = t^{v-1}
  \frac{(v+\sum n_i-2)!}{(v-1)!} \cdot
  \prod_{i = 1}^d \frac{1}{n_i!} \binom{mi-1}i^{n_i}
\end{equation}
for $v=\sum (mi-i-1)n_i +2$. The right-hand side is precisely equal to
$\frac{m-1}{m} v t^{v-1}$ times \eqref{eq:bms}, assuming $d$ large
enough. Removing the spurious term $t$ in $V$, we may interpret
$\frac{m}{m-1}(V-t)$ as the generating function of pointed rooted
constellations, with no weight for the marked vertex. See also the
discussion in~\cite[pp.~130-132]{theseJB}. The generating function of
rooted constellations may be deduced by integrating with respect to
$t$: we perform this computation in Appendix~\ref{sec:rootconst}.

Let us now consider $m$-constellations with a white monochromatic
boundary of prescribed length $mp$, and denote by $F_p^\circ$ the
corresponding generating function, with no weight for the boundary.
It was shown in~\cite{albenque-bouttier}\footnote{In this reference
  the weight per vertex $t$ is set to $1$, but it is not difficult to
  check that it can be taken arbitrary.} that
\begin{equation}
  \frac{dF_p^\circ}{dt} = \binom{mp}{p} V^{(m-1)p}
\end{equation}
and, by integrating and performing the change of variable $t \to V$
using~\eqref{eq:Vdef}, we get
\begin{equation}
  \label{eq:monochrom}
  F_p^\circ = \binom{mp}{p} \left( \frac{V^{(m-1)p+1}}{(m-1)p+1} - \sum_{i=1}^d
    \frac{i}{p+i} \binom{mi-1}{i} x_i V^{(m-1)(p+i)} \right)
\end{equation}
which coincides with~\cite[Proposition~10]{albenque-bouttier} up to an
hypergeometric identity.

Interestingly, it is possible to collect all the
$(F_p^\circ)_{p \geq 0}$ into a single grand generating function which
admits a rational parametrization. This property holds at the more
general level of hypermaps, as we explain in Appendix~\ref{sec:ising}
by restating the results from~\cite[Chapter~8]{eynard} in our present
setting. Let us just state here the result for $m$-constellations:

\begin{prop}
  \label{prop:ratX}
  Let $X(y)$ be the formal Laurent series in $y^{-1}$ defined by
  \begin{equation}
    \label{eq:Xdef}
    X(y) := \sum_{i=1}^d x_i y^{mi-1} + \sum_{p=0}^\infty \frac{F_p^\circ}{y^{mp+1}}
  \end{equation}
  and let $x(z)$ and $y(z)$ be the Laurent polynomials
  \begin{equation}
    \label{eq:xyconst}
    x(z) := z + \sum_{k=1}^d \alpha_k z^{1-mk}, \qquad  y(z) := \frac{V}{z} + z^{m-1},
  \end{equation}
  where the $\alpha_k$ and $V$ are as in Theorem~\ref{thm:main}. Then, we have
  \begin{equation}
    X(y(z)) = x(z)
  \end{equation}
  which should be understood as a substitution of series in $z$,
  $y(z)^{-1}$ being a formal power series in $z$ without constant
  term.
\end{prop}

It may be shown~\cite{AB21} that this proposition is equivalent to the
general expression~\eqref{eq:monochrom} for $(F_p^\circ)_{p \geq 0}$.
Remarkably, the \emph{same} Laurent polynomials $x(z)$ and $y(z)$
provide a rational parametrization for the generating function of
$m$-constellations with a \emph{black} monochromatic boundary (see
again Appendix~\ref{sec:ising} for the derivation):

\begin{prop}
  \label{prop:ratY}
  Let $F_p^\bullet$ denote the generating function of
  $m$-constellations with a black monochromatic boundary of length
  $mp$, $p \geq 1$, with no weight for the boundary. Let $Y(x)$ be the
  formal Laurent series in $x^{-1}$ defined by
  \begin{equation}
    \label{eq:Ydef}
    Y(x) = x^{m-1} + \sum_{p=0}^\infty \frac{F_p^\bullet}{x^{mp+1}}.
  \end{equation}
  Then, we have
  \begin{equation}
    Y(x(z))=y(z)
  \end{equation}
  which should be understood as a substitution of series in $z^{-1}$,
  $x(z)^{-1}$ being a formal power series in $z^{-1}$ without constant
  term.
\end{prop}

\section{A functional equation for hypermaps with an alternating boundary}
\label{sec:funeq}

In this section we consider the general setting of hypermaps: our
purpose is to establish a functional equation determining, for all
$r \geq 0$, the generating function $A_r$ of hypermaps with an
alternating boundary of length $2r$, counted with a weight $t$ per
vertex and weight $t_i$ (resp.\ $\tilde{t}_i$) per black (resp.\
white) inner face of degree $i$ for every $i\geq 1$.

\begin{figure}
  \centering
  \includegraphics[scale=1]{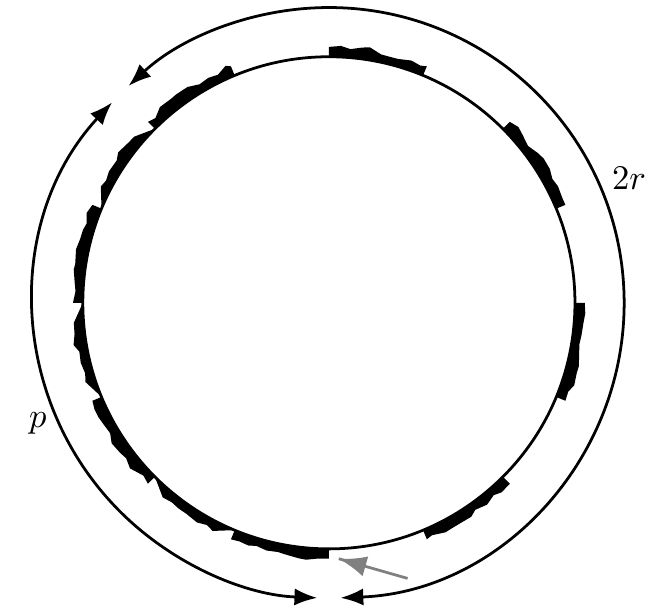}
  \caption{To obtain a functional equation for the enumeration of bicolored maps with an alternating boundary, we consider a more general boundary condition, with a monochromatic part (of length denoted by $p$) and an alternating one (of length denoted by $2r$).}
  \label{fig:layout}
\end{figure}

As it is often the case in map enumeration, it is not obvious how to
write a recursion determining the family $(A_r)_{r \geq 0}$ directly,
but we can obtain a recursion determining the more general family
$(M_{p,r})_{p,r \geq 0}$, where $M_{p,r}$ is the generating function
of hypermaps with ``mixed'' boundary condition of the form
\begin{equation}
  \label{eq:mixed}
  \underbrace{\bullet \bullet \cdots \bullet}_{\text{length } p}
  \underbrace{\circ\bullet\circ\bullet\cdots\circ\bullet}_{\text{length } 2r},
\end{equation}
see Figure~\ref{fig:layout} for an illustration. We then have
$A_r=M_{0,r}$, while $M_{p,0}$ is equal to the generating function
$W_p$ of hypermaps with a black monochromatic boundary of length
$p$. In view of Theorem~\ref{thm:parhyp} in Appendix~\ref{sec:ising}
below, $W_p$ can be considered as known, since (when degrees are
bounded) we have a rational parametrization for the grand generating
function
\begin{equation}
  W(x) := \sum_{p=0}^\infty \frac{W_p}{x^{p+1}}.
\end{equation}
Let us introduce the corresponding generating functions for $A_r$ and $M_{p,r}$:
\begin{equation}
  A(w) := \sum_{r=0}^\infty A_r w^{r+1}, \qquad
  M(x,w) := \sum_{p=0}^\infty \sum_{r=1}^\infty M_{p,r} \frac{w^r}{x^{p+1}}.
\end{equation}
Our reason for working with series in $x^{-1}$ and $w$, and for
shifting some exponents by $1$, is that it will lead to more compact
expressions below. Note that, conventionally, $A_0=M_{0,0}=W_0=t$.

\begin{figure}
  {\centering
    \includegraphics[scale=0.8]{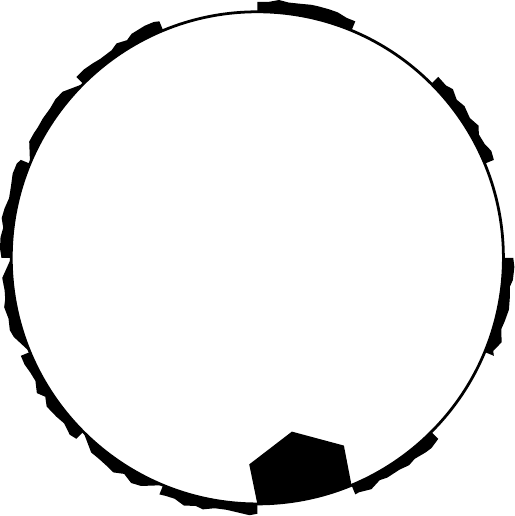}
    \hspace{2em}
    \includegraphics[scale=0.8]{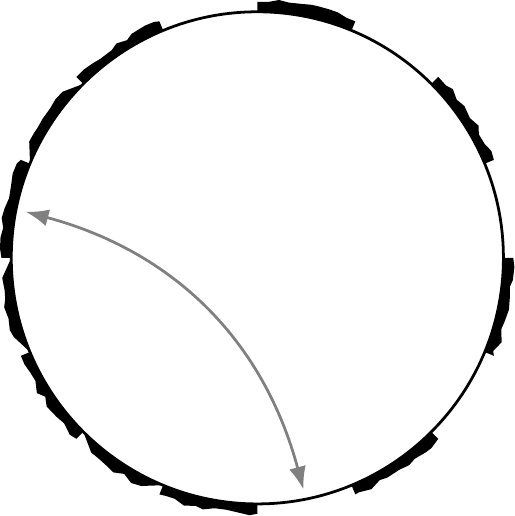}
    \hspace{2em}
    \includegraphics[scale=0.8]{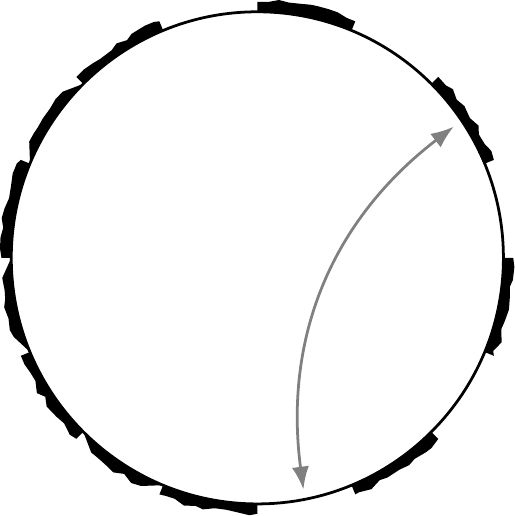} \\
    \vspace{2em}
    \includegraphics[scale=0.8]{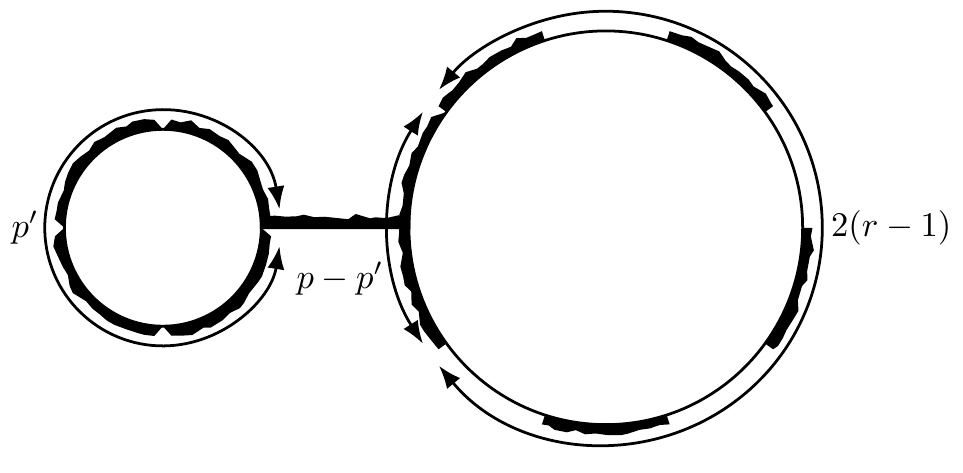}
    \hspace{2em}
    \includegraphics[scale=0.8]{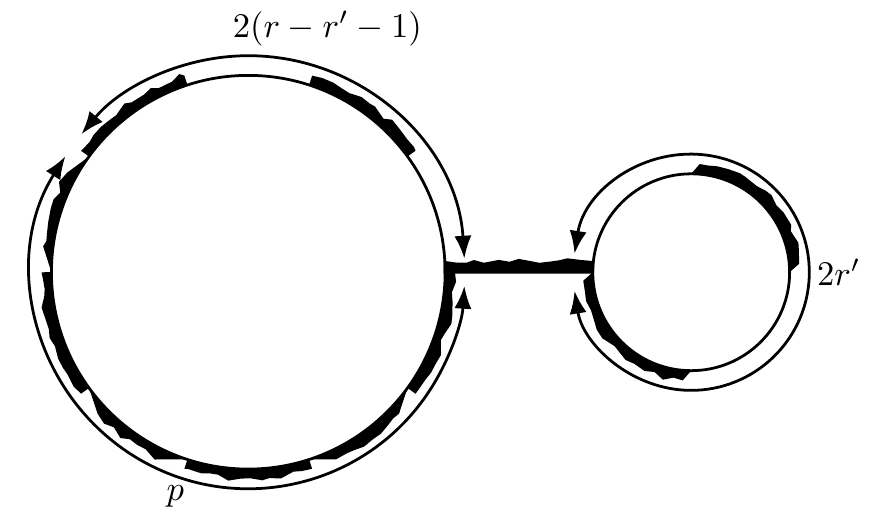}
  }
  \caption{Recursive decomposition of a hypermap with mixed boundary
    condition, as counted by $M_{p,r}$. We ``peel'' the first white
    edge of the alternating part of the boundary, and the three
    situations illustrated in the top row may occur: we may either
    discover a new black inner face (left)---which will increase the
    length of the black monochromatic boundary---or the white edge may
    be identified to a black edge, either in the monochromatic
    (middle) or the alternating (right) boundary. The bottom row
    displays the lengths of the different boundary parts in the second
    and third cases (we exclude by convention the case $r'=r-1$ on the
    right since it is the same as the case $p'=p$ on the left).}
  \label{fig:recurrence}
\end{figure}

\begin{lem}
  For $p \geq 0$ and $r \geq 1$, we have the recursion relation
  \begin{equation}
    M_{p,r} = \sum_{i \geq 1} t_i M_{p+i,r-1} + \sum_{p'=0}^p W_{p'} M_{p-p',r-1} + \sum_{r'=0}^{r-2} A_{r'} M_{p,r-r'-1}.
  \end{equation}
\end{lem}

\begin{proof}
  See Figure~\ref{fig:recurrence}.
\end{proof}

The above recursion translates into the functional equation
\begin{multline}
  M(x,w) = w \sum_{i \geq 1} t_i \left[ x^i (M(x,w) + W(x)) \right]_{x^{<0}} \\ + w x W(x) (M(x,w) + W(x))  + A(w) M(x,w)
\end{multline}
where the notation $[\cdot]_{x^{<0}}$ means that we keep only terms
with negative powers of $x$ (since in the sum over $i$ we ``miss''
some initial terms of $M(x,w)+W(x)$). We may rewrite the functional
equation in the form
\begin{equation}
  \label{eq:kerneleq}
  K(x,w) M(x,w) = R(x,w)
\end{equation}
with
\begin{equation}
  \label{eq:KRdef}
  \begin{split}
  K(x,w) &:= 1 - A(w) - w x Y(x) \\ R(x,w) &:= w x W(x) Y(x) - w \sum_{i \geq 1} t_i
  \left[ x^i (M(x,w)+W(x)) \right]_{x^{\geq 0}}
\end{split}
\end{equation}
where the notation $[\cdot]_{x^{\geq 0}}$ should be self-explanatory
and 
\begin{equation}
  Y(x) :=  \sum_{i \geq 1} t_i x^{i-1} + W(x)
\end{equation}
is the same as in Theorem~\ref{thm:parhyp}.

We recognize a functional equation in one catalytic variable $x$,
which is linear in $M(x,w)$ and thus amenable to the kernel
method. When the $t_i$ vanish for $i$ large enough, we may rewrite the
equation in a form which allows to apply~\cite[Theorem~3]{BoJe06} and
conclude that $M(x,w)$, hence $A(w)$, is algebraic in $w$, $x$, $W(x)$
and the $t_i$. We will however not attempt to work out an explicit
expression for $A(w)$ in general, but we will rather concentrate on
the case of $m$-constellations for which there is a simplification,
and which is sufficient for the application we have in mind.

\section{Simplification in the case of $m$-constellations}
\label{sec:constsimp}

We now specialize the face weights to
\begin{equation}
  t_i =
  \begin{cases}
    1 & \text{if $i=m$,}\\
    0 & \text{otherwise,}
  \end{cases} \qquad
  \tilde{t}_i =
  \begin{cases}
    x_{i/m} &  \text{if $m$ divides $i$,}\\
    0 & \text{otherwise.}
  \end{cases}
\end{equation}
A simplification occurs due to the following:

\begin{lem}
  In the case of constellations, we have $M_{p,r}=0$ unless $m$ divides $p$.
\end{lem}

\begin{proof}
  A hypermap with a boundary can be endowed with a canonical
  orientation of it edges, by orienting each edge in order to have,
  say, white on its right (using the boundary condition for the
  boundary edges). Assume that the hypermap is an $m$-constellation,
  and consider a cycle, not necessarily oriented: we denote by $a$
  (resp.\ $b$) the number of clockwise-oriented (resp.\
  counterclockwise-oriented) edges along it. We claim that $m$ divides
  $a-b$, which may be checked by induction on the number of faces
  encircled by the cycle. Applying this property to the contour of the
  boundary (or rather, of each of its simple components), the result
  follows.
\end{proof}

The above lemma implies that
\begin{equation}
  \label{eq:keyprop}
  M(x,w)+W(x)=\frac{A(w)}{w x}+O\left(\frac{1}{x^{m+1}}\right)
\end{equation}
which allows to rewrite the quantity $R(x,w)$ of~\eqref{eq:KRdef} as
\begin{equation}
  \label{eq:Rsimp}
  R(x,w) = w x W(x) Y(x) - x^{m-1} A(w) = w x Y(x)^2 - w x^m Y(x) - x^{m-1} A(w).
\end{equation}
We now apply the kernel method, or rather a slight variant as we
explain in Remark~\ref{rem:kernel} below. We claim that there exists a
unique formal power series $\omega(x)$ in $x^{-1}$ without constant
term such that
\begin{equation}
  K(x,\omega(x)) = 0.
\end{equation}
Indeed, this condition can be rewritten
\begin{equation}
  \omega(x) = \frac{1-A(\omega(x))}{x Y(x)} = x^{-m} \frac{1-A(\omega(x))}{1 + x^{1-m} W(x)}
\end{equation}
which fixes inductively all the coefficients of $\omega(x)$, that
turns out to be a series in $x^{-m}$. Then, by~\eqref{eq:kerneleq} and
since the substitution $M(x,\omega(x))$ is well-defined, we deduce
that
\begin{equation}
  R(x,\omega(x)) = 0.
\end{equation}
This gives a linear system of two equations determining $\omega(x)$
and $A(\omega(x))$ as
\begin{equation}
  \label{eq:omegaA}
  \omega(x) = \frac{x^{m-2}}{Y(x)^2}, \qquad A(\omega(x)) = 1 - \frac{x^{m-1}}{Y(x)}.
\end{equation}

At this stage, we assume that there exists $d$ such that $x_i=0$ for
$i>d$ (note that we have not used this assumption so far). This allows
to plug in the rational parametrization of $Y(x)$ given by
Proposition~\ref{prop:ratY}. We find
\begin{equation}
  \label{eq:ratom}
  \omega(x(z)) = \frac{x(z)^{m-2}}{y(z)^2}, \qquad A(\omega(x(z))) = 1 - \frac{x(z)^{m-1}}{y(z)}
\end{equation}
with $x(z)$ and $y(z)$ given by~\eqref{eq:xyconst}.  We now observe
that both right-hand sides are actually series in $s=z^{-m}$ and,
setting $w(s)=\omega(x(z))$, we obtain the rational parametrization
announced in Theorem~\ref{thm:main}, which is thereby established.
We conclude this section by some remarks.

\begin{rem}
  Even though we assume that there is a bound on faces degrees in
  Theorem~\ref{thm:main}, we may take formally the limit
  $d \to \infty$: by reversion we find that $A_r$ is a polynomial in
  $V,\alpha_1,\ldots,\alpha_{r+1}$, and each of these quantities is a
  well-defined series in $t,x_1,x_2,\ldots$.
\end{rem}

\begin{rem}
  \label{rem:kernel}
  In the usual kernel method, one would rather look for series
  $\xi(w)$ such that $K(\xi(w),w)=0$. This would certain succeed but
  at the price of some complications, since we would find $m$
  different Puiseux series in $w^{1/m}$, namely the roots of the
  equation $\omega(\xi(w))=w$. The reason why our variant works is
  that $K$ and $R$ involve only one unknown series $A(w)$. This
  property seems specific to the case of $m$-constellations: indeed,
  the simplification hinges upon the relation~\eqref{eq:keyprop},
  which would fail if we allowed for faces of degree not divisible by
  $m$, and upon the fact that the sum over $i$ in~\eqref{eq:KRdef}
  involves only a single nonzero term for $i=m$, which would not be
  the case if we allowed for black faces of degree higher than $m$.
\end{rem}

\begin{figure}
\centering
\includegraphics[width=\textwidth]{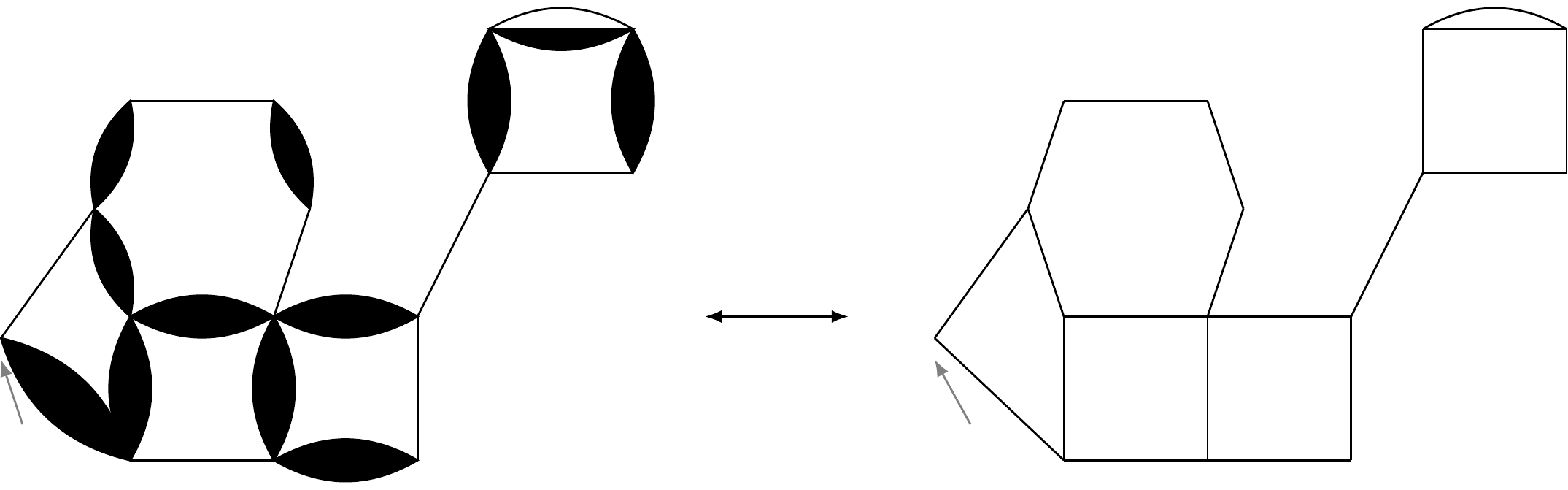}
\caption{Bijection between 2-constellations with an alternating
  boundary and bipartite planar maps with a boundary.}
\label{fig:bipartite}
\end{figure}

\begin{rem}
  \label{rem:bipcase}
  Considering the case $m=2$ provides a nontrivial consistency check.
  Indeed, as displayed on Figure~\ref{fig:bipartite},
  $2$-constellations with an alternating boundary may be identified
  with bipartite planar maps with a boundary, whose enumeration is
  well-known. The rational parametrization~\eqref{eq:ratom} takes the
  form
  \begin{equation}
    A\left(\frac{1}{y(z)^2}\right) = 1 - \frac{x(z)}{y(z)}.
  \end{equation}
  To match this expression with known formulas, we shall
  consider the modified generating function
  $\tilde{A}(y)=\sum_{r \geq 0} \frac{A_r}{y^{2r+1}}=y
  A\left(\frac1{y^{2}}\right)$. It then admits the rational
  parametrization
  \begin{equation}
    \tilde{A}(y(z))=y(z)-x(z)= \frac{V}z - \sum_{k \geq 1} \frac{\alpha_k}{z^{2k-1}}
  \end{equation}
  which matches that given in~\cite[Theorem~3.1.3]{eynard}, up to the
  change of variable $z \to \gamma z$, $V=\gamma^2$.
\end{rem}

\begin{rem}
  Taking $x=x_1$ and $w=\omega(x_2)$ in~\eqref{eq:kerneleq}, and
  using~\eqref{eq:Rsimp} and~\eqref{eq:omegaA}, we find that the
  generating function $M(x,w)$ of $m$-constellations with a mixed
  boundary condition~\eqref{eq:mixed} takes the parametric form
  \begin{equation}
    M(x_1,\omega(x_2)) = \frac1{x_1} \left(\frac{(x_1/x_2)^m-1}{
        x_1 Y(x_1) - x_2 Y(x_2)} (x_2 Y(x_2))^2 + x_1^m
      - x_1 Y(x_1) - x_2 Y(x_2)\right)
  \end{equation}
  and a rational parametrization follows by further substituting
  $x_1=x(z_1)$ and $x_2=x(z_2)$.
\end{rem}

\section{Asymptotics for Eulerian triangulations}
\label{sec: eul trig}

The purpose of this section is to establish Theorem~\ref{thm:asymp}.
Let us denote by $T_{n,r}$ the number of Eulerian triangulations with
a (non necessarily semi-simple) alternating boundary of length $2r$
having $n$ black triangles. We form the generating function
\begin{equation}
  T(t,u):= \sum_{n \geq 0} \sum_{r \geq 0} T_{n,r} t^n u^r.
\end{equation}
In the series $A$ of Theorem~\ref{thm:main}, we rather attach a weight
$t$ to vertices. But, in a triangulation counted by $T_{n,r}$, the
number of vertices $v$ is given by
\begin{equation}
  v = n + r + 1
\end{equation}
(observe that there are $n$ white triangles and $3n+r$ edges, then
apply Euler's relation). Thus, the series $A$ specialized to the case
of Eulerian triangulations satisfies
\begin{equation}
  A(w) = t w T(t,tw)
\end{equation}
and it follows from the case $m=3$ of Corollary~\ref{cor:mangpar},
upon taking $w=w(W/V)$ with $W$ a new variable, that $T$ admits the
rational parametrization
\begin{equation}
  \label{eq:Tparam}
  T\left( V-2V^2, \frac{W(1-2V)(1+VW)}{(1+W)^2}\right) =
  \frac{(1+W)(1-2V-V^2W)}{(1-2V)(1+VW)}.
\end{equation}

Let us now make the connection with the generating function $B(t,z)$
of Eulerian triangulations with a semi-simple alternating boundary, as
defined in Theorem~\ref{thm:asymp}. Considering that a triangulation
with a general boundary can be decomposed into a semi-simple ``core''
(which is the semi-simple component containing the root) and general
components attached to every other vertex on the boundary, we obtain the
substitution relation
\begin{equation}
  T(t,u)=B(t,uT(t,u)).
\end{equation}
Combining with~\eqref{eq:Tparam}, we obtain that $B(t,z)$ admits the
rational parametrization
\begin{equation}
  \label{eq:Bparam}
  B\left( V-2V^2, \frac{W(1-2V-V^2W)}{1+W}\right) =
  \frac{(1+W)(1-2V-V^2W)}{(1-2V)(1+VW)}.
\end{equation}
By eliminating $V$ and $W$, we find that $B(t,z)$ must be a root of
the algebraic equation:
\begin{multline}
16t^3(z-1)^2B^4-32t^2(z-1)^2(t+z)B^3\\
+t(z-1)(24t^2z+32tz^2+16z^3-16t^2-52tz-16z^2-z)B^2\\
-z(z-1)(8t^2-20t-1)(t+z)B+z^2(t+1)^3=0.
\end{multline}
By solving this quartic equation and picking the unique solution which
is a formal power series in $t$ (the others being Laurent series with
negative order), we obtain the expression~\eqref{eq:Bsol} for
$B(t,z)$. The first few terms read
\begin{multline}
  B(t,z) = \left(1+z+z^2+z^3+z^4+z^5+\cdots\right)+t \left(z+2 z^2+3 z^3+4 z^4+5
    z^5+\cdots\right) \\
  +t^2 \left(3 z+8 z^2+15 z^3+24 z^4+35
   z^5+\cdots\right) + t^3 \left(12 z+38 z^2+83 z^3+152 z^4+250
   z^5+\cdots\right) \\
  +t^4 \left(56 z+199 z^2+486 z^3+988 z^4+1790
    z^5+\cdots\right)+ \\
  +t^5 \left(288 z+1112 z^2+2958 z^3+6536 z^4+12822
   z^5+\cdots\right)+\cdots.
\end{multline}

The asymptotics are then deduced by standard methods of analytic
combinatorics. For fixed $z\in (0,1/4)$, the dominant singularity of
$B(t,z)$ is at $t=1/8$ with
\begin{equation}
  B(t,z) - c_0(z) - c_1(z) t \sim
  \frac{2z}{\sqrt{(z-1)(4z-1)^3}} (1-8t)^{3/2}, \qquad t \to \frac18
\end{equation}
for some functions $c_0(z)$ and $c_1(z)$.  We then use the transfer
theorem~\cite[Corollary VI.1]{flajolet-sedgewick} to obtain
\eqref{eq:alternating-gf-better-asymp}, which immediately yields
\eqref{eq:alternating-gf-first-asymp}.  Then, as the function of $z$
appearing in \eqref{eq:alternating-gf-better-asymp} is algebraic and
has a unique dominant singularity at $z=1/4$, we can apply the
transfer theorem once again to get
\eqref{eq:alternating-gf-second-asymp}.  The
asymptotics~\eqref{eq:alternatingZ} for $Z(p)$ are obtained by a
similar reasoning.  All these computations may be checked using the
companion Maple file available on the web page of the second author,
and the proof of Theorem~\ref{thm:asymp} is now complete.

\section{Conclusion}
\label{sec:conc}

In this paper, we have shown how to enumerate hypermaps with an
alternating boundary, explicited the solution in the case of
$m$-constellations, and obtained explicit and asymptotic results in
the case of Eulerian triangulations with a semi-simple alternating
boundary, crucial in the analysis of the layer decomposition of
Eulerian triangulations~\cite{carrance-trig-eul}.

Our approach relies on the framework of functional equations with
catalytic variables arising from the recursive ``peeling''
decomposition of maps. This framework is closely related to the
so-called topological recursion for maps: in a nutshell, the core idea
is that, given a family of maps defined by some ``bulk'' conditions
(degree-dependent weights, etc), the corresponding generating
functions for all sorts of topologies and boundary conditions can be
constructed as functions defined on a common algebraic curve known as
the ``spectral curve''. In the context of hypermaps (or equivalently
the Ising model on maps), this idea is discussed
in~\cite[Chapter~8]{eynard}, and the spectral curve is nothing but the
algebraic curve $(x(z),y(z))_{z \in \mathbb{C} \cup \{\infty\}}$ of
genus zero that we have been using in this paper. Our results show
that alternating boundaries fit nicely in this context, at least in
the case of $m$-constellations, since we see from~\eqref{eq:ratom}
that our series $A(w)$ ``lives'' on the same spectral curve.

Let us now list some possible directions for future research. First,
we have explicited the solution of the functional equation of
Section~\ref{sec:funeq} only in the case of $m$-constellations: it
might be instructive to treat other cases, such as the Ising model on
four-regular maps, which would correspond to taking $t_2=\tilde{t}_2$
and $t_4=\tilde{t}_4$ nonzero, and all others $t_i$ and $\tilde{t}_i$
zero.

Second, one may of course study other types of boundary conditions for
hypermaps. In fact, it is widely believed that, in order to understand
the metric properties of random maps decorated with an Ising model,
one should be able to handle general---or at least
``typical''---boundary conditions. Besides this ambitious research
program, some specific cases are of interest. As discussed
in~\cite[Chapters 8 and 9]{theseAC}, considering alternating
boundaries with defects is relevant to the study of metric balls in
the Upper Half-Plane Eulerian Triangulation, and our approach extends
nicely to this setting. In another direction, if one tries to extend
the layer decomposition to, say, Eulerian quadrangulations, then one
is led to consider boundary conditions which are obtained by repeating
two possible patterns, namely $\circ\bullet$ and
$\mathord{\circ\circ}\mathord{\bullet\bullet}$. Handling these
generalized alternating boundaries represents a new technical
challenge, as one should embed them in a family of peeling-invariant
boundary conditions extending~\eqref{eq:mixed}.

Finally, as mentioned in the introduction, and in view of the quite
simple rational parametrization of Corollary~\ref{cor:mangpar}, it is
natural to ask whether our results could be derived by a bijective
approach.

\appendix

\section{The generating function of rooted $m$-constellations}
\label{sec:rootconst}

Let $C$ denote the generating function of rooted $m$-constellations
(i.e., $m$-constellations with a white monochromatic boundary of
arbitrary length), counted with a weight $t$ per vertex and a weight
$x_i$ per white face of degree $mi$ for every $i \geq 1$ (including
the boundary face). As discussed in Section~\ref{sec:reminders}, we
have
\begin{equation}
  \label{eq:Cintform}
  C = \frac{m}{m-1} \int (V-t) dt.
\end{equation}
Here and in the following, there are no integration constants since
$[t^0]C=0$.  On the other hand, recalling the definition of $A_1$ in
Theorem~\ref{thm:main}, we have
\begin{equation}
  C = A_1 - t^2,
\end{equation}
since there is an obvious bijection between rooted $m$-constellations
and $m$-constellations with boundary condition $\bullet\circ$ not
reduced to a single edge. From the rational parametrization, we find
that
\begin{equation}
  A_1 = V^2 - (2m-3) V \alpha_1 + \frac{m-1}2 \left( (m-2) \alpha_1^2-2\alpha_2 \right).
\end{equation}
Using
\begin{equation}
  t=V-(m-1)\alpha_1,
\end{equation}
we deduce that
\begin{equation}
  \label{eq:Cnoint}
  C = \frac{m \alpha_1}2 \left(2V - (m-1)\alpha_1\right) - (m-1)(V \alpha_1+\alpha_2).
\end{equation}
Let us check the consistency of this expression
with~\eqref{eq:Cintform}. For this, we write
\begin{equation}
  \int V dt = \int V dV - (m-1) \int V d\alpha_1
\end{equation}
and hence
\begin{equation}
  \begin{split}
    \frac{m}{m-1} \int (V-t) dt &= \frac{m}{2(m-1)} (V^2 - t^2) - m \int V d\alpha_1 \\
    &= \frac{m}2 \alpha_1 \left(2V - (m-1)\alpha_1\right) - m \int V
    d\alpha_1.
  \end{split}
\end{equation}
Thus, comparing with~\eqref{eq:Cnoint}, we should have
\begin{equation}
  m \int V d\alpha_1 = (m-1) (V \alpha_1 + \alpha_2).
\end{equation}
This is easy to check using the definition~\eqref{eq:alphadef} of the
$\alpha_k$, performing the integration term by term and doing simple
manipulations of binomial coefficients.

\begin{rem}
  The expression~\eqref{eq:Cnoint} is ``canonical'' as
  $V,\alpha_1,\ldots,\alpha_d$ are algebraically independent. Indeed,
  the formal variables $t,x_1,\ldots,x_d$ are by definition
  independent and they can be recovered polynomially from
  $V,\alpha_1,\ldots,\alpha_d$.
\end{rem}

\section{Rational parametrizations for monochromatic boundaries}
\label{sec:ising}

\paragraph{Hypermaps and the Ising model.}

We consider generating functions of hypermaps, counted with a weight
$t$ per vertex and a weight $t_i$ (resp.\ $\tilde{t}_i$) per black
(resp.\ white) inner face of degree $i$ for every $i=2,\ldots,d$
(resp.\ $i=2,\ldots,\tilde{d}$), where $d$ and $\tilde{d}$ are fixed
integers larger than $1$. In order to stick with the setting
of~\cite{eynard}, inner faces of degree one are forbidden, but there
would be no problem allowing them.

We denote by $W_p$ (resp.\ $\tilde{W}_p$) the generating function of
hypermaps with a black (resp.\ white) monochromatic boundary of length
$p$, $p \geq 0$. We have conventionally $W_0=\tilde{W}_0=t$.

\begin{thm}[{reformulation of~\cite[Theorem~8.3.1]{eynard}}]
  \label{thm:parhyp}
  Let $Y(x)$ and $X(y)$ be the formal Laurent series in $x^{-1}$ and
  $y^{-1}$, respectively, defined by
  \begin{equation}
    \label{eq:YXhyp}
    Y(x) := \sum_{i=2}^d t_i x^{i-1} + \sum_{p=0}^\infty \frac{W_p}{x^{p+1}}, \qquad
    X(y) := \sum_{i=2}^{\tilde{d}} \tilde{t}_i y^{i-1} + \sum_{p=0}^\infty \frac{\tilde{W}_p}{y^{p+1}}
  \end{equation}
  and let $x(z)$ and $y(z)$ be the Laurent polynomials
  \begin{equation}
    \label{eq:xyhyp}
    x(z) := z + \sum_{k=0}^{\tilde{d}-1} a_k z^{-k}, \qquad
    y(z) := \frac{V}z + \sum_{k=0}^{d-1} b_k z^k
  \end{equation}
  where $V,a_0,\ldots,a_{\tilde{d}-1},b_0,\ldots,b_{d-1}$ are the
  series in
  $t,t_2,\ldots,t_d,\tilde{t}_2,\ldots,\tilde{t}_{\tilde{d}}$
  determined by the conditions
  \begin{equation}
    \label{eq:condhyp}
    y(z) - \sum_{i=2}^d t_i x(z)^{i-1} = t z^{-1} + O(z^{-2}), \qquad
    x(z) - \sum_{i=2}^{\tilde{d}} \tilde{t}_i y(z)^{i-1} = \frac{t z}{V} + O(z^2).
  \end{equation}
  Then, we have
  \begin{equation}
    X(y(z))=x(z), \qquad Y(x(z))=y(z)
  \end{equation}
  which should be understood as substitutions of series in $z$ and
  $z^{-1}$ respectively, $y(z)^{-1}$ (resp.\ $x(z)^{-1}$) being a
  formal power series in $z$ (resp.\ $z^{-1}$) without constant term.
\end{thm}

Since we have reformulated slightly the statement given
in~\cite{eynard}, let us explain the connection. Eynard does not quite
consider hypermaps, but rather planar maps whose faces have degree at
least $3$ and carry Ising ($+$ or $-$) spins, and where each edge
receives a weight $c_{++}$ (resp.\ $c_{--}$, $c_{+-}$) if it is
incident to two $+$ faces (resp.\ two $-$ faces, one $+$ and one $-$
face). There is a standard one-to-many correspondence between such
maps and hypermaps, obtained by ``collapsing'' the bivalent (degree
$2$) faces of the hypermaps and attaching a spin $+$ (resp.\ $-$) to
the remaining black (resp.\ white) faces. In this correspondence, two
adjacent $+$ faces were necessarily separated by an odd number of
bivalent faces, with one more white than black bivalent face. Thus,
the effective weight per $++$ edge is
\begin{equation}
  c_{++} = \frac{\tilde{t}_2}{1-t_2 \tilde{t}_2}.
\end{equation}
Similarly, for the $--$ and $+-$ edges we get effective weights
\begin{equation}
  c_{--} = \frac{t_2}{1-t_2 \tilde{t}_2}, \qquad c_{+-} = \frac{1}{1-t_2 \tilde{t}_2}.
\end{equation}
All these relations can be put in matrix form
\begin{equation}
  \begin{pmatrix}
    c_{++} & c_{+-} \\ c_{+-} & c_{--}
  \end{pmatrix} =
  \begin{pmatrix}
    -t_2 & 1 \\ 1 & -\tilde{t}_2
  \end{pmatrix}^{-1}
\end{equation}
and we get the following identification with the notations
from~\cite[Theorem~8.1.1]{eynard} (which deals with yet another
reformulation of hypermaps):
\begin{equation}
  c=-1, \qquad a=-t_2, \qquad b=-\tilde{t}_2.
\end{equation}
Plugging into~\cite[Theorem~8.3.1]{eynard}, $Y(x)$ matches our present
definition~\eqref{eq:YXhyp}, and we obtain the form~\eqref{eq:xyhyp}
for $x(z)$ and $y(z)$ through the change of variable $z \to z/\gamma$,
with $\gamma^2=V$. The relation $X(y(z))=x(z)$ is obtained by
exchanging the roles of white and black faces, and reparametrizing
$z \to V/z$.

\begin{rem}
  Theorem~\ref{thm:parhyp} can be given a combinatorial
  interpretation, as will be explained in~\cite{AB21}.
\end{rem}

\paragraph{Specialization to constellations.}

Let us change the $d$ and $\tilde{d}$ of Theorem~\ref{thm:parhyp} into
$m$ and $md$, respectively, and set the face weights to
\begin{equation}
  t_i =
  \begin{cases}
    1 & \text{if $i=m$,}\\
    0 & \text{otherwise,}
  \end{cases} \qquad
  \tilde{t}_i =
  \begin{cases}
    x_{i/m} &  \text{if $m$ divides $i$,}\\
    0 & \text{otherwise.}
  \end{cases}
\end{equation}
It is not difficult to check that, with such weights, the
conditions~\eqref{eq:condhyp} imply that the series $a_k$ and $b_k$
vanish unless $m$ divides $k+1$. Thus, $x(z)$ and $y(z)$ can be put in
the form~\eqref{eq:xyconst}, the first condition~\eqref{eq:condhyp}
implying that $b_{m-1}=1$. By extracting the coefficient of $z^{1-km}$
for $k \geq 1$ in the second condition~\eqref{eq:condhyp}, we find
that $\alpha_k=a_{mk-1}$ is given by~\eqref{eq:alphadef}, while
extracting the coefficient of $z$ yields an equation equivalent
to~\eqref{eq:Vdef}. Identifying $F^\circ_p=\tilde{W}_{mp}$ and
$F^\bullet_p=W_{mp}$, we obtain Propositions~\ref{prop:ratX} and
\ref{prop:ratY} as specializations of Theorem~\ref{thm:parhyp}.



\vspace{-0.75em}

\begin{thebibliography}{10}
\vspace{-0.5em}
  
\bibitem{albenque-bouttier}
M.~Albenque and J.~Bouttier.
\newblock Constellations and multicontinued fractions: application to
  {E}ulerian triangulations.
\newblock In {\em 24th {I}nternational {C}onference on {F}ormal {P}ower
  {S}eries and {A}lgebraic {C}ombinatorics ({FPSAC} 2012)}, Discrete Math.
Theor. Comput. Sci. Proc., AR, pages 805--816, 2012.

\bibitem{AB21}
M.~Albenque and J.~Bouttier.
\newblock On the slice decomposition of planar hypermaps.
\newblock In preparation, 2021.

\bibitem{ams}
M.~Albenque, L.~M{\'e}nard, and G.~Schaeffer.
\newblock Local convergence of large random triangulations coupled with an
Ising model.
\newblock {\em Trans. Am. Math. Soc.}, 374(1):175--217, 2021.

\bibitem{angel}
O.~Angel.
\newblock Growth and percolation on the uniform infinite planar triangulation.
\newblock {\em Geom. Funct. Anal.}, 13:935–974, 2003.

\bibitem{bernardi-fusy-2020}
O. Bernardi and \'{E}. Fusy.
\newblock Unified bijections for planar hypermaps with general cycle-length
  constraints.
\newblock {\em Ann. Inst. Henri Poincar\'{e} D}, 7(1):75--164, 2020.

\bibitem{BoKa87}
D.~V. Boulatov and V.~A. Kazakov.
\newblock {T}he {I}sing model on a random planar lattice: the structure of the
  phase transition and the exact critical exponents.
\newblock {\em Phys. Lett. B}, 186(3-4):379--384, 1987.

\bibitem{BoJe06}
M. Bousquet-M\'elou and A. Jehanne.
\newblock Polynomial equations with one catalytic variable, algebraic series
  and map enumeration.
\newblock {\em J. Combin. Theory Ser. B}, 96(5):623--672, 2006.

\bibitem{bm-s}
M.~Bousquet-M\'{e}lou and G.~Schaeffer.
\newblock Enumeration of planar constellations.
\newblock {\em Adv. in Appl. Math.}, 24(4):337--368, 2000.

\bibitem{BoSc02_Ising}
M. Bousquet-M{\'e}lou and G. Schaeffer.
\newblock The degree distribution in bipartite planar maps: applications to the
{I}sing model.
\newblock \href{https://arxiv.org/abs/math/0211070}{arXiv:math/0211070}, 2002.

\bibitem{bdg}
J.~Bouttier, P.~Di~Francesco, and E.~Guitter.
\newblock Planar maps as labeled mobiles.
\newblock {\em Electr. J. Combin.}, 11(1), 2004.

\bibitem{fomap}
J.~Bouttier, P.~Di~Francesco, and E.~Guitter.
\newblock {B}locked edges on {E}ulerian maps and mobiles: application to
  spanning trees, hard particles and the {I}sing model.
\newblock {\em J. Phys. A}, 40(27):7411--7440, 2007.

\bibitem{bouttier-fusy-guitter}
J.~Bouttier, É.~Fusy, and E.~Guitter.
\newblock On the two-point function of general planar maps and hypermaps.
\newblock {\em Ann. Inst. Henri Poincaré Comb. Phys. Interact.}, 1:265--306,
  2014.

\bibitem{theseJB}
J. Bouttier.
\newblock {\em Physique statistique des surfaces al{\'e}atoires et combinatoire
  bijective des cartes planaires}.
\newblock Doctoral thesis, {Universit{\'e} Pierre et Marie Curie -- Paris 6},
  June 2005. \url{https://tel.archives-ouvertes.fr/tel-00010651}.

\bibitem{carrance-trig-eul}
A.~Carrance.
\newblock Convergence of Eulerian triangulations.
\newblock \href{https://arxiv.org/abs/1912.13434}{arXiv:1912.13434}, 2019.

\bibitem{theseAC}
A.~Carrance.
\newblock {\em {Random colored triangulations}}.
\newblock Doctoral thesis, {Universit{\'e} de Lyon}, September 2019. \url{https://tel.archives-ouvertes.fr/tel-02338972}.

\bibitem{BC20}
G. Chapuy and M. {Do{\l}{\k{e}}ga}.
\newblock Non-orientable branched coverings, $b$-Hurwitz numbers, and
positivity for multiparametric Jack expansions.
\newblock \href{https://arxiv.org/abs/2004.07824}{arXiv:2004.07824}, 2020.

\bibitem{ct}
L.~Chen and J.~Turunen.
\newblock Critical Ising model on random triangulations of the disk:
enumeration and local limits.
\newblock {\em Commun. Math. Phys.}, 374(3):1577--1643, 2020.

\bibitem{ct2}
L.~Chen and J.~Turunen.
\newblock Ising model on random triangulations of the disk: phase transition.
\newblock \href{https://arxiv.org/abs/2003.09343}{arXiv:2003.09343}, 2020.

\bibitem{CurienSFnotes}
N. Curien.
\newblock Peeling random planar maps.
\newblock Lecture notes of the 2019 Saint-Flour Probability Summer School,
  preliminary version available
  at~\url{https://www.math.u-psud.fr/~curien/enseignement.html}, 2019.

\bibitem{Douglas91}
M.~R.~Douglas.
\newblock The two-matrix model.
\newblock In {\em Random surfaces and quantum gravity ({C}arg\`ese, 1990)},
  volume 262 of {\em NATO Adv. Sci. Inst. Ser. B Phys.}, pages 77--83. Plenum,
  New York, 1991.

\bibitem{eynard}
B.~Eynard.
\newblock {\em Counting surfaces}, volume~70 of {\em Progress in Mathematical
  Physics}.
\newblock Birkh\"{a}user/Springer, [Cham], 2016.
\newblock CRM Aisenstadt chair lectures.

\bibitem{flajolet-sedgewick}
P.~Flajolet and R.~Sedgewick.
\newblock {\em Analytic combinatorics}.
\newblock Cambridge University Press, Cambridge, 2009.

\bibitem{Kazakov86}
V.~A. Kazakov.
\newblock Ising model on a dynamical planar random lattice: exact solution.
\newblock {\em Phys. Lett. A}, 119(3):140--144, 1986.

\bibitem{lando-zvonkine}
S.~Lando and A.~Zvonkin.
\newblock {\em Graphs on surfaces and their applications}, volume 141 of {\em
  Encyclopaedia of Mathematical Sciences}.
\newblock Springer-Verlag, Berlin, 2004.
\newblock With an appendix by Don B. Zagier, Low-Dimensional Topology, II.

\bibitem{Prodinger04}
H. Prodinger.
\newblock The kernel method: a collection of examples.
\newblock {\em S\'{e}m. Lothar. Combin.}, B50f, 19, 2004.

\bibitem{Schaeffer15}
G. Schaeffer.
\newblock Planar maps.
\newblock In {\em Handbook of enumerative combinatorics}, Discrete Math. Appl.
  (Boca Raton), pages 335--395. CRC Press, Boca Raton, FL, 2015.

\bibitem{turunen20}
J.~Turunen.
\newblock Interfaces in the vertex-decorated Ising model on random
triangulations of the disk.
\newblock \href{https://arxiv.org/abs/2003.11012}{arXiv:2003.11012}, 2020.

\bibitem{Tutte68}
W.~T. Tutte.
\newblock On the enumeration of planar maps.
\newblock {\em Bull. Amer. Math. Soc.}, 74:64--74, 1968.

\end{thebibliography}
\end{document}